\documentclass[a4paper, 11pt]{amsart}
\usepackage{amsmath,amsthm,amssymb}
\usepackage{graphicx,color}
\usepackage[margin=2cm]{geometry}
\usepackage{tikz,enumerate}

\theoremstyle{plain}
\newtheorem{theorem}{Theorem}[section]
\newtheorem{proposition}[theorem]{Proposition}

\numberwithin{equation}{section}

\theoremstyle{definition}
\newtheorem{definition}[theorem]{Definition}

\newtheorem{example}[theorem]{Example}

\newcommand{\C}{\mathbb{C}}
\newcommand{\Q}{\mathbb{Q}}
\newcommand{\R}{\mathbb{R}}
\newcommand{\Z}{\mathbb{Z}}

\newcommand{\CP}{\mathbb{C}P}
\newcommand{\fF}{\mathfrak{F}}
\newcommand{\cA}{\mathcal{A}}
\newcommand{\cF}{\mathcal{F}}

\newcommand{\cS}{\mathcal{S}}
\newcommand{\cT}{\mathcal{T}}
\newcommand{\cZ}{\mathcal{Z}}

\newcommand{\cP}{\mathcal{P}}
\newcommand{\cQ}{\mathcal{Q}}

\DeclareMathOperator{\bideg}{bideg}
\DeclareMathOperator{\Tor}{Tor}

\DeclareMathOperator{\conv}{conv}
\DeclareMathOperator{\codim}{codim}
\DeclareMathOperator{\ord}{ord}

\begin{document}
\title[C-rigid polytope not B-rigid]{Example of C-rigid polytopes which are not B-rigid}

\author[S.~Choi]{Suyoung Choi}
\address{Department of mathematics, Ajou University,
206, World cup-ro, Yeongtong-gu, Suwon, 16499, Republic of Korea}
\email{schoi@ajou.ac.kr}
\thanks{The first named author was supported by Basic Science Research Program through the National Research Foundation of Korea(NRF) funded by the Ministry of Science, ICT \& Future Planning(NRF-2016R1D1A1A09917654).}

\author[K.~Park]{Kyoungsuk Park}
\address{Department of mathematics, Ajou University,
206, World cup-ro, Yeongtong-gu, Suwon, 16499, Republic of Korea}
\email{bluemk00@ajou.ac.kr}




\date{\today}

\begin{abstract}
A simple polytope $P$ is said to be \emph{B-rigid} if its combinatorial structure is characterized by its Tor-algebra, and is said to be \emph{C-rigid} if its combinatorial structure is characterized by the cohomology ring of a quasitoric manifold over $P$.
It is known that a B-rigid simple polytope is C-rigid.
In this paper, we, further, show that the B-rigidity is not equivalent to the C-rigidity.
\end{abstract}

\maketitle


\section{Introduction}
Let $P$ be a simple polytope of dimension $n$.
A closed, smooth manifold of dimension $2n$ is called a \emph{quasitoric manifold} over $P$ if it admits a locally standard action of the $n$-dimensional torus $T^n$ whose orbit space can be identified with $P$ (see \cite{Davis-Januszkiewicz1991} for more details).
One typical example of quasitoric manifolds is a complex projective space $\CP^n = \C^{n+1}\setminus \{O\} / \sim$ where $(x_0, x_1, \ldots, x_n) \sim (rx_0, rx_1, \ldots, rx_n)$ for all non-zero real numbers $r$.
One can see that a $T^n$-action on $\C^{n+1} \setminus \{O\}$ defined by
$$
    (t_1, \ldots, t_n) \cdot (x_0, x_1, \ldots, x_n) = (x_0, t_1 x_1, \ldots, t_n x_n )
$$ induces a locally standard $T^n$-action on $\CP^n$, and its orbit space is an $n$-dimensional simplex $\Delta^n$.
Hence, $\CP^n$ is a quasitoric manifold over $\Delta^n$.
Here is one naive question: can $\CP^n$ be a quasitoric manifold over a different simple polytope $Q$ other than $\Delta^n$?
The answer is ``no''.
It is well-known that for a quasitoric manifold $M$ over $P$, the Betti numbers of $M$ coincide with the $h$-numbers of $P$.
Since $\Delta^n$ is the only simple polytope whose $h$-vector is $(1,1, \ldots, 1)$ which is the sequence of the Betti numbers of $\CP^n$, only $\Delta^n$ can be an orbit space of $\CP^n$.

This phenomenon raises one fundamental question which asks how much combinatorial information of $P$ is decided by the topology of a quasitoric manifold $M$ over $P$.
Masuda and Suh \cite{Masuda-Suh2008} dealt with it as the property of simple polytope.
Throughout this paper, $H^\ast(X)$ denotes the integral cohomology ring of a topological space $X$.

\begin{definition}
  A simple polytope $P$ is said to be \emph{(toric) cohomologically rigid} (or \emph{C-rigid}) if there is no simple polytope $Q \not\approx P$ such that $H^{\ast}(M)\cong H^{\ast}(N)$ for some quasitoric manifolds $M$, $N$ over $P$, $Q$, respectively.
\end{definition}

We note that there are many simple polytopes which do not support any quasitoric manifold.
Since the ``C-rigidity'' requires the existence of quasitoric manifold over a given polytope, there is no canonical way to define the C-rigidity of such polytopes.
Thus there has been confusion; some literatures such as \cite{Choi-Panov-Suh2010} define that such polytope is not C-rigid, but some literatures such as \cite{BEMPP} define that such polytopes are C-rigid as their conventions.
However, it does not matter because the C-rigidity of $P$ should be considered only when $P$ supports a quasitoric manifold.

One important step on the theory of rigidity of simple polytopes is to find some combinatorial invariants of $P$ determined by the cohomology ring of a quasitoric manifold over $P$.
Let $P$ be an $n$-dimensional simple polytope with $m$ facets and $\cA=\Q[v_1,\ldots,v_m]$ the polynomial ring over $\Q$ with $\deg v_i =2$ for all $i$.
We consider $\Q$ as an $\cA$-module via the map $\cA\rightarrow\Q$ sending each $v_i$ to $0$.
The \emph{Stanley-Reisner ring} of $P$ is the quotient ring $\Q(P)=\Q[v_1,\ldots,v_m]\slash I_P,$ where $I_P$, the \emph{Stanley-Reisner ideal} of $P$, is the homogeneous ideal generated by all square-free monomials $v_{i_1}\cdots v_{i_r}$ such that $F_{i_1}\cap\cdots\cap F_{i_r}=\emptyset$.
We also regard $\Q(P)$ as an $\cA$-module.

Let $\Lambda[u_1,\ldots,u_m]$ be an exterior algebra.
Then, we have a differential bigraded algebra $R=\Lambda[u_1,\ldots,u_m]\otimes \cA$ with map $d \colon R\rightarrow R$, where $\bideg u_i=(-1,2)$, $\bideg v_i=(0,2)$, $du_i=v_i$, $dv_i=0$ and $\otimes$ is a tensor product over $\Q$.
Then, $R$ is a free $\cA$-module.
Let $R^{-i}=\Lambda^i[u_1,\ldots,u_m]\otimes \cA$, where $\Lambda^i [u_1, \ldots, u_m]$ is the submodule of $\Lambda[u_1,\ldots,u_m]$ spanned by monomials of length $i$.
Then, we have the free resolution of $\Q$, known as the \emph{Koszul resolution}, as follows:
$$
  0 \rightarrow R^{-m} \stackrel{d}{\rightarrow} \cdots \stackrel{d}{\rightarrow} R^{-1} \stackrel{d}{\rightarrow} \cA \stackrel{d}{\rightarrow} \Q \rightarrow 0.
$$
By taking $\otimes_{\cA}\Q(P)$ to the Koszul resolution, we obtain the Tor-module of $P$
$$\Tor(P) := H (R\otimes_{\cA}\Q(P)) = H (\Lambda [u_1,\ldots,u_m] \otimes \Q(P)).$$
Furthermore, it has the natural algebra structure induced from the Koszul resolution, so it is called the \emph{Tor-algebra} $\Tor(P)$ of $P$.
Since $\Tor(P)$ has the bigraded structure, the bigraded Betti number $\beta^{-i,2j}(P)$ is also defined for $P$.

It is shown in \cite[Lemma~3.7]{Choi-Panov-Suh2010} that, for two quasitoric manifolds $M$ and $N$ over $P$ and $Q$, respectively, if $H^\ast(M)\cong H^\ast(N)$ as graded rings, then $\Tor(P) \cong \Tor(Q)$ as bigraded rings.
This fact stimulates to consider the hierarchy of rigidities of simple polytopes (see \cite{Choi-Kim2011} and \cite{Buchstaber2008}).

\begin{itemize}
  \item $P$ is \emph{combinatorially rigid} (or \emph{A-rigid}) if there is no $n$-dimensional simple polytope $Q\not\approx P$ such that $\beta^{-i,2j}(P)=\beta^{-i,2j}(Q)$ for all $i,j$.
  \item $P$ is \emph{B-rigid} if there is no $n$-dimensional simple polytope $Q\not\approx  P$ such that $\Tor(P)\cong \Tor(Q)$ as bigraded rings.
\end{itemize}

In addition, we have the following implications.
\begin{enumerate}
  \item If $P$ is A-rigid, then $P$ is B-rigid.
  \item If $P$ supports a quasitoric manifold and $P$ is B-rigid, then $P$ is C-rigid.
\end{enumerate}

It is natural to ask whether the converse of the above statements hold or not.
In \cite{Choi2013}, the first named author found the counterexample of the reverse implication of (1); there is a $3$-dimensional B-rigid simple polytope with $11$ facets which is not A-rigid.
However, the reverse implication of (2) has been open (see the remark in Section~3 of \cite{BEMPP}).

In this paper, we shall provide a counterexample of the reverse implication of (2), that is, there is a C-rigid simple polytope which is not B-rigid while it supports a quasitoric manifold.
More precisely, in Section~\ref{sec:C-rigid but B-rigid}, we provide two distinct simple polytopes $\cP$ and $\cQ$ of dimension $5$ having $8$ facets satisfying the following:
\begin{enumerate}[(a)]
  \item both $\cP$ and $\cQ$ support quasitoric manifolds,
  \item $\Tor(\cP) \cong \Tor(\cQ)$ as bigraded rings,
  \item there is no other polytope whose Tor-algebra is isomorphic that of $\cP$, and
  \item no two quasitoric manifolds $M$ over $\cP$ and $N$ over $\cQ$ have the isomorphic cohomology rings, that is, $H^\ast(M) \not\cong H^\ast(N)$ as graded rings.
\end{enumerate}

It proves that the A-, B-, and C-rigidities are not equivalent to each other under the condition that $P$ supports a quasitoric manifold.

This paper is mainly based on a part of the Ph.D. thesis of the second named author \cite{Park2015_thesis} supervised by the first named author.

\section{Simple polytopes with a few facets}
In this section, we recall some useful facts on $n$-dimensional simple polytopes with $n+3$ facets.
The structure of $n$-simple polytope $P$ with $n+3$ facets is well-known (see, for example, \cite{Grunbaum2003}): $P$ can be obtained by a sequence of wedge operations from either a cube or the dual of cyclic polytope $C^\ast_{2k-2,2k+1}$ for some $k\geq 2$.
If $P$ is obtained from a cube by a sequence of wedge operations, then $P$ is the product of three simplices.
It is shown in \cite[Theorem~5.3]{Choi-Panov-Suh2010} that every product of simplices is A-rigid, and so is $P$.
Hence, in the remain of the section, we only consider the case where $P$ is obtained from a cyclic polytope.
It is convenient to represent $P$ by the Gale diagram in $\R^2$.

For $k\geq 2$, let $P_{2k+1}$ be a regular $(2k+1)$-gon in $\R^2$ with center at the origin $O$ with the vertex set $[2k+1] = \{1,2,\ldots, 2k+1\}$ in counterclockwise order.
For a given surjective map $\phi \colon \fF=\{F_1, \ldots, F_{n+3}\} \rightarrow [2k+1]$, we construct a simplicial complex $K$ on $[n+3]$ by
$$I (\subset [n+3]) \mbox{ is a simplex of } K \iff O \in \conv\{\phi(F_i) \mid i\in[n+3] \setminus I \}.$$
It is known that $K$ is a boundary complex of some simple $n$-polytope $P$ with the facet set $\fF$.
One observes that the combinatorial structure of $P$ only depends on $|\phi^{-1}(1)|, \ldots, |\phi^{-1}(2k+1)|$ up to rotating and reflecting of $P_{2k+1}$.
Hence, an $n$-dimensional simple polytope $P$ with $n+3$ facets, which is not a product of three simplices, is representable on $P_{2k+1}$ with the assigned numbers $[a_1, \ldots, a_{2k+1}]$ where $a_1 + \cdots + a_{2k+1} = n+3$ up to rotation and reflection.

Now, let us compute the bigraded Betti numbers of $P$ represented on $P_{2k+1}$ with assigned numbers $[a_1, \ldots, a_{2k+1}]$.
We firstly remark that every simple polytope has the following duality on its bigraded Betti numbers.
\begin{proposition}
  Let $Q$ be an $n$-dimensional simple polytope with $m$ facets. Then
  \begin{enumerate}
    \item $\beta^{-i, 2j}(Q) = 0$ if $i<0$ or $i> m-n$,
    \item $\beta^{0,0}(Q)=\beta^{-(m-n),2m}(Q)=1$, and
    \item $\beta^{-i,2j}(Q)=\beta^{-(m-n)+i,2(m-j)}(Q)$.
  \end{enumerate}
\end{proposition}
By the above proposition, we have $\beta^{-1,2j}(P)=\beta^{-2,2(n+3-j)}(P)$, and, hence, $\{\beta^{-1,2j}(P) \mid j \geq 2 \}$ completely determines all bigraded Betti numbers of $P$.
It should be noted that $\beta^{-1,2j}$ is equal to the number of degree $2j$ monomial elements in a minimal basis of the Stanley-Reisner ideal $I_P$.
Thus it is enough to find the minimal monomial basis of $I_P$.
By definitions of $I_P$ and a Gale-diagram, we have the following proposition.
\begin{proposition}\label{prop:generator_in_gale}
  Let $P$ be an $n$-dimensional simple polytope with $n+3$ facets $F_1,\ldots,F_{n+3}$.
  Let $v_1,\ldots,v_{n+3}$ be indeterminates corresponding to the facets of $P$ and $\phi$ the corresponding map of $P$.
  Then, the followings are equivalent:
  \begin{enumerate}
    \item $v_{i_1}\cdots v_{i_\ell}$ is a monomial generator of $I_P$,
    \item $F_{i_1}\cap\cdots\cap F_{i_\ell} =\emptyset$,
    \item $O \notin \conv\{ \phi(F_j) \mid j\in[n+3]\setminus \{i_1,\ldots,i_\ell\} \}$
  \end{enumerate}
\end{proposition}
Moreover the monomial $v_{i_1}\cdots v_{i_\ell}$ is an element of the minimal monomial basis if and only if $v_{i_1}\cdots v_{i_\ell} \in I_P$ and $F_J = \bigcap_{j\in J} F_j \ne \emptyset$ for any $J \subsetneq \{i_1,\ldots,i_\ell\}$.
Now consider the collection $\cS$ of subsets of $[n+3]$ which satisfies (3) in Proposition~\ref{prop:generator_in_gale}, then $\cS$ has a natural poset structure by an inclusion. Note that $\cS$ is nonempty and finite. So we can choose the set $\cT$ of the minimal elements of $\cS$.
Then $\cT$ supports the minimal monomial basis of $I_P$.
If $P$ is represented on $P_{2k+1}$ with assigned numbers $[a_1, a_2, \ldots, a_{2k+1}]$, then we can assign each $i \in [2k+1]$ by each element $e_i = \{ i_1, \ldots, i_\ell \} \subset [n+3]$ of $\cT$ such that $\phi(F_{j}) \in \{i, i+1, \ldots, i+k-1\}$ modulo $2k+1$ for all $j \in e_i$ and $\phi(F_{j'}) \not\in \{i, \ldots, i+k-1\}$ modulo $2k+1$ for all $j' \not\in e_i$.
We, therefore, have
$$\beta^{-1,2j}(P) = |\{ i \in [2k+1] \mid |e_i|=j \} = | \{ i \in [2k+1] \mid a_i + a_{i+1} + \cdots + a_{i+k-1} = j \}|,$$
where $a_{2k+2} =a_1, a_{2k+3}=a_2, \ldots, a_{3k}= a_{k-1}$.

On the other hand, Erokhovets \cite{Erokhovets2011} showed that the Tor-algebra $\Tor(P)$ is completely determined by the bigraded Betti numbers of $P$. Indeed, $\Tor(P)$ is isomorphic to the cohomology ring of the moment-angle complex $\cZ_P$ (see \cite{Buchstaber-Panov2015}), and it is known in \cite{Lopez1989} that
$$
    \cZ_P \simeq \#_{i=1}^{2k-1} S^{2 \phi_i-1} \times S^{2 \varphi_{i +k - 1}-2},
$$
where $\phi_i = a_i + a_{i+1} + \cdots + a_{i+k-2}$ and $\varphi_i = a_i + a_{i+1} + \cdots + a_{i+k-1}$ for all $i \in [2k+1]$, and indices are taken module $2k+1$.
We remark that the A-rigidity of $P$ is equivalent to the B-rigidity, that is, if $P$ is B-rigid, then $P$ is A-rigid as well. However, $P$ is rarely to be A-rigid (and B-rigid) as seen in \cite{Bosio_A-rigid}.

In particular, we, therefore, note that two simple polytopes represented on $P_5$ with assigned numbers $[a,b,c,d,e]$ and $[a',b',c',d',e']$ have isomorphic Tor-algebras if and only if $\{a+b,b+c,c+d,d+e,e+a\}=\{a'+b',b'+c',c'+d',d'+e',e'+a'\}$ as multi-sets.
We further have the following criterion to find all polytopes whose Tor-algebras are isomorphic to that of a given polytope represented on $P_5$.

\begin{proposition} \label{prop:criterion_on_the_same_betti}
  Let $P$ be the simple $n$-polytope with $n+3$ facets represented on $P_5$ with assigned number $[a,b,c,d,e]$.
  We also let $Q$ be another simple $n$-polytope with $n+3$ facets represented on $P_5$ with assigned number $[a',b',c',d',e']$.
Then, the following are equivalent:
\begin{enumerate}
  \item $\Tor(P) \cong \Tor(Q)$ as bigraded rings,
  \item $P$ and $Q$ have the same bigraded Betti numbers, and
  \item $[a',b',c',d',e']$ appears as a cyclic subgraph with $5$ vertices of the Peterson graph with assigned numbers as in Figure~\ref{peterson}.
\end{enumerate}
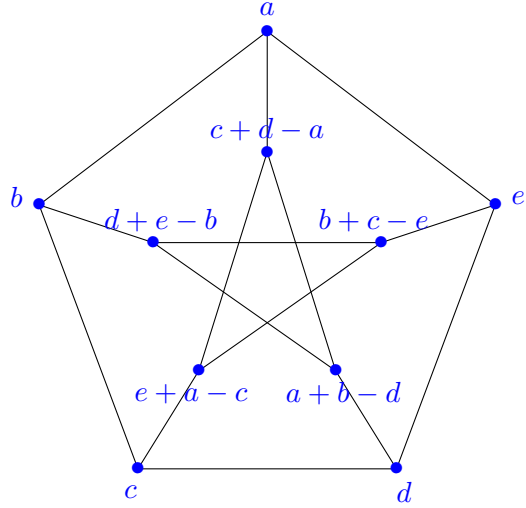
\begin{figure}
    \center
    \begin{tikzpicture}
      \draw (1,4.5) to (4,6.8); \draw (4,6.8) to (7,4.5); \draw (7,4.5) to (5.7,1); \draw (5.7,1) to (2.3,1); \draw (2.3,1) to (1,4.5);
      \draw (1,4.5) to (2.5,4); \draw (2.5,4) to (5.5,4); \draw (5.5,4) to (7,4.5);
      \draw (4,6.8) to (4,5.2);
      \draw (2.3,1) to (3.1,2.3); \draw (4,5.2) to (3.1,2.3);
      \draw (5.7,1) to (4.9,2.3); \draw (4,5.2) to (4.9,2.3);
      \draw (2.5,4) to (4.9,2.3); \draw (5.5,4) to (3.1,2.3);
      \textcolor{blue}{\node at (4,6.8) {$\bullet$}; \node at (1,4.5) {$\bullet$}; \node at (2.3,1) {$\bullet$}; \node at (5.7,1) {$\bullet$}; \node at (7,4.5) {$\bullet$}; \node at (4,5.2) {$\bullet$}; \node at (2.5,4) {$\bullet$}; \node at (3.1,2.3) {$\bullet$}; \node at (4.9,2.3) {$\bullet$}; \node at (5.5,4) {$\bullet$};
      }
      \textbf{\textcolor{blue}{\node at (4,7.1) {$a$};
      \node at (0.7,4.6) {$b$};
      \node at (2.2,0.7) {$c$};
      \node at (5.8,0.7) {$d$};
      \node at (7.3,4.6) {$e$};
      \node at (4,5.5) {$c+d-a$};
      \node at (2.6,4.3) {$d+e-b$};
      \node at (3,2) {$e+a-c$};
      \node at (5,2) {$a+b-d$};
      \node at (5.4,4.3) {$b+c-e$};
      }}

    \end{tikzpicture}
  \caption{Peterson graph: criterion for Gale diagrams on a pentagon to have the isomorphic Tor-algebras}\label{peterson}
  \end{figure}
\end{proposition}
\begin{proof}
    It is enough to show that (2) and (3) are equivalent.
  We first show that the set of the sum of the adjacent 2 vertices of each subgraph with 5 vertices in the above Peterson graph is $\{a+b,b+c,c+d,d+e,e+a\}$.
  We consider 4 cases for the sequence of vertices as follow:
  \begin{itemize}
    \item[(1)] $[a,b,c,d,e]$ or reverse order,
    \item[(2)] kinds of $[a,e,d,a+b-d,c+d-a]$ or reverse order,
    \item[(3)] kinds of $[a,e,b+c-e,e+a-c,c+d-a]$ or reverse order, and
    \item[(4)] $[a+b-d,d+e-b,b+c-e,e+a-c,c+d-a]$ or reverse order.
  \end{itemize}
  One can easily check that for each case we have $\{a+b,b+c,c+d,d+e,e+a\}$ for the set of the sum of the adjacent 2 vertices, and, hence, we have $24$ sequences from the Peterson graph.

We then show that there are $24$ sequences $[a',b',c',d',e']$ satisfying
  $$\{a+b,b+c,c+d,d+e,e+a\}=\{a'+b',b'+c',c'+d',d'+e',e'+a'\}.$$
  Consider the system of equations
  $$a'+b'=A,\,b'+c'=B,\,c'+d'=C,\,d'+e'=D,\,e'+a'=E,$$
  where $\{A,B,C,D,E\}=\{a+b,b+c,c+d,d+e,e+a\}$.
  Then, we have $5!=120$ systems, and there, thus, can be at most $24$ sequences $[a',b',c',d',e']$ up to rotation.
  Therefore, each sequence $[a',b',c',d',e']$ satisfying $\{a+b,b+c,c+d,d+e,e+a\}=\{a'+b',b'+c',c'+d',d'+e',e'+a'\}$ is uniquely obtained from a cyclic subgraph with $5$ vertices  of the Peterson graph with assigned numbers as in Figure~\ref{peterson}.
\end{proof}

\begin{example} \label{ex:[3,1,2,1,1]}
    We consider the simple polytope $\cP$ represented on $P_5$ with assigned numbers $[3,1,2,1,1]$.
    Then the corresponding diagram as in Proposition~\ref{prop:criterion_on_the_same_betti} is the following.

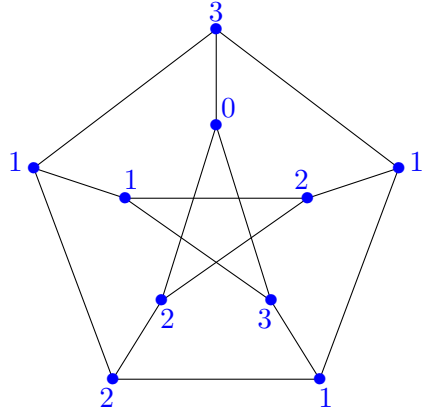
\begin{center}
      \begin{tikzpicture}[scale=.8]
      \draw (1,4.5) to (4,6.8); \draw (4,6.8) to (7,4.5); \draw (7,4.5) to (5.7,1); \draw (5.7,1) to (2.3,1); \draw (2.3,1) to (1,4.5);
      \draw (1,4.5) to (2.5,4); \draw (2.5,4) to (5.5,4); \draw (5.5,4) to (7,4.5);
      \draw (4,6.8) to (4,5.2);
      \draw (2.3,1) to (3.1,2.3); \draw (4,5.2) to (3.1,2.3);
      \draw (5.7,1) to (4.9,2.3); \draw (4,5.2) to (4.9,2.3);
      \draw (2.5,4) to (4.9,2.3); \draw (5.5,4) to (3.1,2.3);
      \textcolor{blue}{\node at (4,6.8) {$\bullet$}; \node at (1,4.5) {$\bullet$}; \node at (2.3,1) {$\bullet$}; \node at (5.7,1) {$\bullet$}; \node at (7,4.5) {$\bullet$}; \node at (4,5.2) {$\bullet$}; \node at (2.5,4) {$\bullet$}; \node at (3.1,2.3) {$\bullet$}; \node at (4.9,2.3) {$\bullet$}; \node at (5.5,4) {$\bullet$};
      }
      \textbf{\textcolor{blue}{\node at (4,7.1) {$3$};
      \node at (0.7,4.6) {$1$};
      \node at (2.2,0.7) {$2$};
      \node at (5.8,0.7) {$1$};
      \node at (7.3,4.6) {$1$};
      \node at (4.2,5.5) {$0$};
      \node at (2.6,4.3) {$1$};
      \node at (3.2,2) {$2$};
      \node at (4.8,2) {$3$};
      \node at (5.4,4.3) {$2$};
      }}

        \end{tikzpicture}
\end{center}

    One can easily see that there are only two types of cyclic subgraphs with $5$ vertices whose assigned numbers are all positive integers:
    $[3,1,2,1,1]$ and $[2,2,2,1,1]$ up to rotating and reflecting elements.
    Let $\cQ$ be the simple polytope represented on $P_5$ with assigned numbers $[2,2,1,1,1]$.
    Hence, $\cQ$ is the only polytope whose Tor-algebra is isomorphic to that of $\cP$ with $\cP \not\approx \cQ$.
\end{example}

%
\section{Main theorem} \label{sec:C-rigid but B-rigid}
In this section, we give an example of C-rigid simple polytope which is not B-rigid, hence we show that the C-rigidity of a given simple polytope $P$ does not guarantee the B-rigidity of $P$.
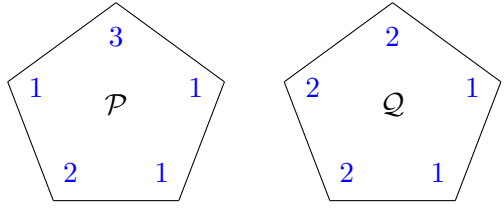
\begin{figure}
  \center
  \begin{tikzpicture}[scale=1.5]
    \tikzstyle{Element} = [draw,circle, inner sep=0pt]
    \textcolor{blue}{\node at (5,3.6) {3};
    \node at (4.3,3.15) {1};
    \node at (4.6,2.4) {2};
    \node at (5.4,2.4) {1};
    \node at (5.7,3.15) {1};}
    \draw (5,3.9) to (4.05,3.2);
    \draw (4.05,3.2) to (4.45,2.15);
    \draw (4.45,2.15) to (5.55,2.15);
    \draw (5.55,2.15) to (5.95,3.2);
    \draw (5.95,3.2) to (5,3.9);
    \node at (5,3.0) {$\cP$};
    \normalsize
  \end{tikzpicture}\quad\quad
  \begin{tikzpicture}[scale=1.5]
    \tikzstyle{Element} = [draw,circle, inner sep=0pt]
    \textcolor{blue}{\node at (5,3.6) {2};
    \node at (4.3,3.15) {2};
    \node at (4.6,2.4) {2};
    \node at (5.4,2.4) {1};
    \node at (5.7,3.15) {1};}
    \draw (5,3.9) to (4.05,3.2);
    \draw (4.05,3.2) to (4.45,2.15);
    \draw (4.45,2.15) to (5.55,2.15);
    \draw (5.55,2.15) to (5.95,3.2);
    \draw (5.95,3.2) to (5,3.9);
    \node at (5,3.0) {$\cQ$};
    \normalsize
  \end{tikzpicture}
  \caption{Gale diagrams of $\cP$ and $\cQ$ : examples of C-rigid but not B-rigid polytopes} \label{fig:Examples of C-rigid but not B-rigid polytopes}
\end{figure}
More precisely, we will show that two simple polytopes $\cP$ and $\cQ$ represented on $P_5$ with assigned numbers $[3,1,2,1,1]$ and $[2,2,2,1,1]$ as in Figure~\ref{fig:Examples of C-rigid but not B-rigid polytopes}, respectively, satisfy the following;

\begin{enumerate}[(a)]
  \item both $\cP$ and $\cQ$ support quasitoric manifolds,
  \item $\Tor(\cP) \cong \Tor(\cQ)$ as bigraded rings,
  \item there is no other polytope whose Tor-algebra is isomorphic that of $\cP$, and
  \item no two quasitoric manifolds $M$ over $\cP$ and $N$ over $\cQ$ have the isomorphic cohomology rings, that is, $H^\ast(M) \not\cong H^\ast(N)$ as graded rings.
\end{enumerate}

If $\cP$ and $\cQ$ hold the above (a)--(d), then both $\cP$ and $\cQ$ are indeed our desired examples.
The statement (b) implies that $\cP$ and $\cQ$ are not B-rigid, and both (c) and (d) imply the C-rigidity of $\cP$ and $\cQ$ under the condition (a) which confirms the existence of supporting quasitoric manifolds.

The statement (a) immediately follows the following theorem due to Erokhovets \cite{Erokhovets2011} (cf. \cite{Gretenkort-Kleinschmidt-Sturmfels1990}).

\begin{theorem}
  Let $P$ be a simple polytope represented on $P_{2k+1}$. Then, $P$ supports a quasitoric manifold if and only if $k \leq 3$.
\end{theorem}

The statements (b) and (c) are already showed in Example~\ref{ex:[3,1,2,1,1]}.
In the remain of the section, let us show the statement (d) which would be the most difficult part.

\begin{theorem}
    The polytopes $\cP$ and $\cQ$ do not support quasitoric manifolds having the isomorphic $\Z$-cohomology rings.
\end{theorem}
\begin{proof}
    Let $P$ be an $n$-dimensional simple polytope with the set of facets $\fF = \{F_1, \ldots, F_m \}$.
    We recall the general fact on quasitoric manifolds due to \cite{Davis-Januszkiewicz1991} that each quasitoric manifold $M$ over $P$ is assigned by the characteristic map $\lambda \colon \fF \to \Z^n$ satisfying that whenever $F_{i_1} \cap \cdots \cap F_{i_\ell} \neq \emptyset$, the set $\{\lambda(F_{i_1}), \ldots, \lambda(F_{i_\ell})\}$ of integral vectors is unimodular.
    We denote $\lambda = (\Lambda_{i,j})_{1 \leq i \leq n, 1 \leq j \leq m}$ by an $n \times m$ integer matrix $\Lambda$ such that the $i$th column of $\Lambda$ is $\lambda(F_i)$.
    It should be noted that we may assume that the first $n$ columns of $\Lambda$ form an identity matrix of size $n$.

    Furthermore, the cohomology of $M$ associated to $\lambda$ is isomorphic to
$$
    \Z[v_1, \ldots, v_m]/ I_P + J,
$$ where $\deg v_i =2$ for all $i=1,\ldots, m$,  $I_P$ is the ideal generated by the square free monomials $v_{i_1} \cdots v_{i_r}$ such that $F_{i_1} \cap \cdots \cap F_{i_r} = \emptyset$, and $J$ is the ideal generated by $n$ linear terms $\Lambda_{i,1} v_1 + \cdots + \Lambda_{i,m} v_m$ for $ 1\leq i \leq n$.

Now, let us suppose there are quasitoric manifolds $M$ over $\cP$ and $N$ over $\cQ$ such that $H^\ast(M) \cong H^\ast(N)$ as graded rings.
Then, $H^\ast(M;\Z_2) = H^{\ast}(M) \otimes_\Z \Z_2 \cong H^{\ast}(N)\otimes_\Z \Z_2 = H^\ast(N;\Z_2)$  as graded rings.
We remark that $H^{\ast}(M; \Z_2)$ is determined by $\Z_2$-characteristic map $\lambda^\R \colon \cF \stackrel{\lambda}{\to} \Z^n \stackrel{\text{mod $2$}}{\longrightarrow} \Z_2^n$. A $\Z_2$-characteristic map also can be represented by an $n \times m$ $\Z_2$-matrix $\Lambda^\R$, called the $\Z_2$-\emph{characteristic matrix}, whose $i$th column is $\lambda^\R(F_i)$.
It should also be remarked that the first $n$ columns of $\Lambda^\R$ form an identity $\Z_2$-matrix.
From now on, we shall show that $H^{\ast}(M;\Z_2)$ is not isomorphic to $H^{\ast}(N;\Z_2)$ as graded rings for any pairs of $\Z_{2}$-characteristic maps over $\cP$ and over $\cQ$, which contradicts to the assumption, so we prove the theorem.

Here is the list of $\Z_{2}$-characteristic matrices over $\cP$.
We give a re-labeling and an order of the facet set of $\cP$ as $\fF=(F_{1_1} , F_{1_2} , F_{1_3} , F_{3_1} , F_{3_2} , F_5, F_2, F_4 )$ such that the face structure of $\cP$ is determined by a surjective map $\phi \colon \fF \to [5]$ defined by $\phi(F_{1_j})=1$ for $j=1,2,3$, $\phi(F_{3_j})=3$ for $j=1,2$, and $\phi(F_i)=i$ for $i=2,4,5$, where $[5]$ is the vertex set of a regular pentagon $P_5$ as a Gale-diagram.
In this order, we consider all possible $\Z_2$-characteristic matrices over $\cP$.
The following is the list of $A_i$'s where $\Lambda^\R = (I_5 | A_i)$ is a $\Z_{2}$-characteristic matrix over $\cP$:

\small
\noindent $A_1 = \left(\begin{array}{ccc}
    1 & 0 & 1 \\
    1 & 0 & 1 \\
    1 & 0 & 1 \\
    0 & 1 & 1 \\
    0 & 1 & 1 \\
\end{array}\right)$
$A_2 = \left(\begin{array}{ccc}
    1 & 0 & 1 \\
    1 & 0 & 1 \\
    1 & 1 & 0 \\
    0 & 1 & 1 \\
    0 & 1 & 1 \\
\end{array}\right)$
$A_3 = \left(\begin{array}{ccc}
    1 & 0 & 1 \\
    1 & 1 & 0 \\
    1 & 0 & 1 \\
    0 & 1 & 1 \\
    0 & 1 & 1 \\
\end{array}\right)$
$A_4 = \left(\begin{array}{ccc}
    1 & 0 & 1 \\
    1 & 1 & 0 \\
    1 & 1 & 0 \\
    0 & 1 & 1 \\
    0 & 1 & 1 \\
\end{array}\right)$
$A_5 = \left(\begin{array}{ccc}
    1 & 1 & 0 \\
    1 & 0 & 1 \\
    1 & 0 & 1 \\
    0 & 1 & 1 \\
    0 & 1 & 1 \\
\end{array}\right)$
$A_6 = \left(\begin{array}{ccc}
    1 & 1 & 0 \\
    1 & 0 & 1 \\
    1 & 1 & 0 \\
    0 & 1 & 1 \\
    0 & 1 & 1 \\
\end{array}\right)$
$A_7 = \left(\begin{array}{ccc}
    1 & 1 & 0 \\
    1 & 1 & 0 \\
    1 & 0 & 1 \\
    0 & 1 & 1 \\
    0 & 1 & 1 \\
\end{array}\right)$
$A_8 = \left(\begin{array}{ccc}
    1 & 1 & 0 \\
    1 & 1 & 0 \\
    1 & 1 & 0 \\
    0 & 1 & 1 \\
    0 & 1 & 1 \\
\end{array}\right)$
$A_9 = \left(\begin{array}{ccc}
    1 & 0 & 1 \\
    1 & 0 & 1 \\
    1 & 0 & 1 \\
    0 & 1 & 1 \\
    1 & 1 & 1 \\
\end{array}\right)$
$A_{10} = \left(\begin{array}{ccc}
    1 & 1 & 0 \\
    1 & 1 & 0 \\
    1 & 1 & 0 \\
    0 & 1 & 1 \\
    1 & 0 & 1 \\
\end{array}\right)$
$A_{11} = \left(\begin{array}{ccc}
    1 & 0 & 1 \\
    1 & 0 & 1 \\
    1 & 0 & 1 \\
    1 & 1 & 1 \\
    0 & 1 & 1 \\
\end{array}\right)$
$A_{12} = \left(\begin{array}{ccc}
    1 & 1 & 0 \\
    1 & 1 & 0 \\
    1 & 1 & 0 \\
    1 & 0 & 1 \\
    0 & 1 & 1 \\
\end{array}\right)$
$A_{13} = \left(\begin{array}{ccc}
    1 & 0 & 1 \\
    1 & 0 & 1 \\
    1 & 0 & 1 \\
    1 & 1 & 1 \\
    1 & 1 & 1 \\
\end{array}\right)$
$A_{14} = \left(\begin{array}{ccc}
    1 & 1 & 0 \\
    1 & 1 & 0 \\
    1 & 1 & 0 \\
    1 & 0 & 1 \\
    1 & 0 & 1 \\
\end{array}\right)$
$A_{15} = \left(\begin{array}{ccc}
    1 & 1 & 0 \\
    1 & 1 & 0 \\
    1 & 1 & 1 \\
    1 & 0 & 1 \\
    1 & 0 & 1 \\
\end{array}\right)$
$A_{16} = \left(\begin{array}{ccc}
    1 & 1 & 0 \\
    1 & 1 & 1 \\
    1 & 1 & 0 \\
    1 & 0 & 1 \\
    1 & 0 & 1 \\
\end{array}\right)$
$A_{17} = \left(\begin{array}{ccc}
    1 & 1 & 0 \\
    1 & 1 & 1 \\
    1 & 1 & 1 \\
    1 & 0 & 1 \\
    1 & 0 & 1 \\
\end{array}\right)$
$A_{18} = \left(\begin{array}{ccc}
    1 & 1 & 1 \\
    1 & 1 & 0 \\
    1 & 1 & 0 \\
    1 & 0 & 1 \\
    1 & 0 & 1 \\
\end{array}\right)$
$A_{19} = \left(\begin{array}{ccc}
    1 & 1 & 1 \\
    1 & 1 & 0 \\
    1 & 1 & 1 \\
    1 & 0 & 1 \\
    1 & 0 & 1 \\
\end{array}\right)$
$A_{20} = \left(\begin{array}{ccc}
    1 & 1 & 1 \\
    1 & 1 & 1 \\
    1 & 1 & 0 \\
    1 & 0 & 1 \\
    1 & 0 & 1 \\
\end{array}\right)$
$A_{21} = \left(\begin{array}{ccc}
    1 & 1 & 1 \\
    1 & 1 & 1 \\
    1 & 1 & 1 \\
    1 & 0 & 1 \\
    1 & 0 & 1 \\
\end{array}\right)$
\normalsize

Here is the list of $\Z_{2}$-characteristic matrices over $\cQ$.
We give a re-labeling and an order of the facet set of $\cQ$ as $\fF=(F_{1_1} , F_{1_2} , F_{2_1}, F_{3_1} , F_{3_2} , F_5, F_{2_2}, F_4 )$ such that the face structure of $\cQ$ is determined by a surjective map $\phi \colon \fF \to [5]$ defined by $\phi(F_{i_j})=1$ for $i=1,2,3$ and $j=1,2$, and $\phi(F_i)=i$ for $i=4,5$.
In this order, we consider all possible $\Z_2$-characteristic matrices over $\cQ$.
The following is the list of $B_i$'s where $(I_5 | B_i)$ is a $\Z_{2}$-characteristic matrix over $\cQ$:

\small
\noindent $B_1 = \left(\begin{array}{ccc}
    1 & 0 & 1 \\
    1 & 0 & 1 \\
    0 & 1 & 0 \\
    0 & 1 & 1 \\
    0 & 1 & 1 \\
\end{array}\right)$
$B_2 = \left(\begin{array}{ccc}
    1 & 0 & 1 \\
    1 & 1 & 0 \\
    0 & 1 & 0 \\
    0 & 1 & 1 \\
    0 & 1 & 1 \\
\end{array}\right)$
$B_3 = \left(\begin{array}{ccc}
    1 & 1 & 0 \\
    1 & 0 & 1 \\
    0 & 1 & 0 \\
    0 & 1 & 1 \\
    0 & 1 & 1 \\
\end{array}\right)$
$B_4 = \left(\begin{array}{ccc}
    1 & 1 & 0 \\
    1 & 1 & 0 \\
    0 & 1 & 0 \\
    0 & 1 & 1 \\
    0 & 1 & 1 \\
\end{array}\right)$
$B_5 = \left(\begin{array}{ccc}
    1 & 0 & 1 \\
    1 & 0 & 1 \\
    0 & 1 & 0 \\
    0 & 1 & 1 \\
    1 & 1 & 1 \\
\end{array}\right)$
$B_6 = \left(\begin{array}{ccc}
    1 & 1 & 0 \\
    1 & 1 & 0 \\
    0 & 1 & 0 \\
    0 & 1 & 1 \\
    1 & 0 & 1 \\
\end{array}\right)$
$B_7 = \left(\begin{array}{ccc}
    1 & 0 & 1 \\
    1 & 0 & 1 \\
    0 & 1 & 0 \\
    1 & 1 & 1 \\
    0 & 1 & 1 \\
\end{array}\right)$
$B_8 = \left(\begin{array}{ccc}
    1 & 1 & 0 \\
    1 & 1 & 0 \\
    0 & 1 & 0 \\
    1 & 0 & 1 \\
    0 & 1 & 1 \\
\end{array}\right)$
$B_9 = \left(\begin{array}{ccc}
    1 & 0 & 1 \\
    1 & 0 & 1 \\
    0 & 1 & 0 \\
    1 & 1 & 1 \\
    1 & 1 & 1 \\
\end{array}\right)$
$B_{10} = \left(\begin{array}{ccc}
    1 & 1 & 0 \\
    1 & 1 & 0 \\
    0 & 1 & 0 \\
    1 & 0 & 1 \\
    1 & 0 & 1 \\
\end{array}\right)$
$B_{11} = \left(\begin{array}{ccc}
    1 & 1 & 0 \\
    1 & 1 & 0 \\
    0 & 1 & 1 \\
    1 & 0 & 1 \\
    1 & 0 & 1 \\
\end{array}\right)$
$B_{12} = \left(\begin{array}{ccc}
    1 & 1 & 0 \\
    1 & 1 & 1 \\
    0 & 1 & 0 \\
    1 & 0 & 1 \\
    1 & 0 & 1 \\
\end{array}\right)$
$B_{13} = \left(\begin{array}{ccc}
    1 & 1 & 0 \\
    1 & 1 & 1 \\
    0 & 1 & 1 \\
    1 & 0 & 1 \\
    1 & 0 & 1 \\
\end{array}\right)$
$B_{14} = \left(\begin{array}{ccc}
    1 & 1 & 1 \\
    1 & 1 & 0 \\
    0 & 1 & 0 \\
    1 & 0 & 1 \\
    1 & 0 & 1 \\
\end{array}\right)$
$B_{15} = \left(\begin{array}{ccc}
    1 & 1 & 1 \\
    1 & 1 & 0 \\
    0 & 1 & 1 \\
    1 & 0 & 1 \\
    1 & 0 & 1 \\
\end{array}\right)$
$B_{16} = \left(\begin{array}{ccc}
    1 & 1 & 1 \\
    1 & 1 & 1 \\
    0 & 1 & 0 \\
    1 & 0 & 1 \\
    1 & 0 & 1 \\
\end{array}\right)$
$B_{17} = \left(\begin{array}{ccc}
    1 & 1 & 1 \\
    1 & 1 & 1 \\
    0 & 1 & 1 \\
    1 & 0 & 1 \\
    1 & 0 & 1 \\
\end{array}\right)$
$B_{18} = \left(\begin{array}{ccc}
    1 & 0 & 1 \\
    1 & 0 & 1 \\
    1 & 1 & 0 \\
    0 & 1 & 1 \\
    0 & 1 & 1 \\
\end{array}\right)$
$B_{19} = \left(\begin{array}{ccc}
    1 & 0 & 1 \\
    1 & 0 & 1 \\
    1 & 1 & 0 \\
    0 & 1 & 1 \\
    1 & 1 & 1 \\
\end{array}\right)$
$B_{20} = \left(\begin{array}{ccc}
    1 & 0 & 1 \\
    1 & 0 & 1 \\
    1 & 1 & 0 \\
    1 & 1 & 1 \\
    0 & 1 & 1 \\
\end{array}\right)$
$B_{21} = \left(\begin{array}{ccc}
    1 & 0 & 1 \\
    1 & 0 & 1 \\
    1 & 1 & 0 \\
    1 & 1 & 1 \\
    1 & 1 & 1 \\
\end{array}\right)$
\normalsize

The above lists can be found by hand-computation or using computer algorithm such as \cite[Algorithm~4.1]{Garrison-Scott2003}.

We remark that the $\Z_2$-cohomology rings of $M$ over $\cP$ associated to $A_i$ can be written by
$$H^\ast(M;\Z_2) = \Z_2 [v_{1_1}, v_{1_2}, v_{1_3}, v_{3_1}, v_{3_2}, v_5, v_2, v_4] / I_\cP + J \cong \Z_2 [x,y,z]/ I_A$$
where $x = v_5$, $y=v_2$ and $z=v_4$.

The following table is the list of generators of $I_A$ with respect to each $A_i$.
$$\begin{array}{|l|l|}
    \hline A_{1} & (x^3 y +yz^3 , y^3 + yz^2 , y^2 z + z^3 , xz, x^4)\\
    \hline A_{2},A_{3},A_{5} & (x^3 y+ x^2 y^2 + y^2 z^2 , y^3 + yz^2 , y^2 z + z^3 , xz , x^4 + x^3 y)\\
    \hline A_{4},A_{6},A_{7} & (x^3 y + y^3 z , y^3 + yz^2 , y^2 z + z^3 , xz , x^4 + x^2 y^2)\\
    \hline A_{8} & (x^3 y+ x^2 y^2 + y^4, y^3 + yz^2, y^2 z+ z^3 , xz, x^4 +y^4)\\
    \hline A_{9},A_{11} & (x^3 y+ yz^3, xy^2+y^3+yz^2,y^2z+z^3,xz,x^4)\\
    \hline A_{10},A_{12} & (x^3y+xy^3+y^4,xy^2+y^z+yz^2,yz^2+z^3,xz,x^4+y^4)\\
    \hline A_{13} & (x^3y+yz^3,x^2y+y^3+yz^2,y^2z+z^3,xz,x^4)\\
    \hline A_{14} & (x^2y^2+xy^3+y^4,x^2y+yz^2,z^3,xz,x^4+ y^4)\\
    \hline A_{15},A_{16},A_{18} & (x^4+y^4+y^3z,x^2y+yz^2,z^3,xz,x^4+x^2y^2+xy^3)\\
    \hline A_{17},A_{19}, A_{20} & (x^4+y^4+y^2z^2,x^2y+yz^2,z^3,xz,x^4+x^2y^2+xy^3)\\
    \hline A_{21} & (x^4+y^4+y^3z+y^2z^2,x^2y+yz^2,z^3,xz,x^4+x^2y^2+xy^3)\\ \hline
\end{array}$$

Similarly, the $\Z_2$-cohomology rings of $N$ over $\cQ$ associated to $B_i$ can also be written by
$$H^\ast(N;\Z_2) = \Z_2 [v_{1_1}, v_{1_2}, v_{2_1}, v_{3_1}, v_{3_2}, v_5, v_{2_2}, v_4] / I_\cP + J \cong \Z_2 [x,y,z]/ I_B$$
where $x = v_5$, $y=v_{2_2}$ and $z=v_4$.
The following table is the list of generators of $I_B$ with respect to each $B_i$.
$$\begin{array}{|l|l|}
    \hline B_{1} & (x^{2}y^{2}+y^{2}z^{2},  y^{4}+y^{2}z^{2},  y^{2}z+z^{3},  xz,  x^{3})\\
    \hline B_{2},B_{3} & (x^{2}y^{2}+xy^{3}+y^{3}z,  y^{4}+y^{2}z^{2},  y^{2}z+z^{3},  xz,  x^{3}+x^{2}y)\\
    \hline B_{4} &(x^{2}y^{2}+y^{4},  y^{4}+y^{2}z^{2},  y^{2}z+z^{3},  xz,  x^{3}+xy^{2})\\
    \hline B_{5},B_{7},B_{18} &(x^{2}y^{2}+y^{2}z^{2},  xy^{3}+y^{4}+y^{2}z^{2},  y^{2}z+z^{3},  xz,  x^{3})\\
    \hline B_{6},B_{8} & (x^{2}y^{2}+y^{4},  xy^{3}+y^{3}z+y^{2}z^{2},  yz^{2}+z^{3},  xz,  x^{3}+xy^{2})\\
    \hline B_{9},B_{19},B_{20} & (x^{2}y^{2}+y^{2}z^{2},  y^{4},  y^{2}z+z^{3},  xz,  x^{3})\\
    \hline B_{10} & (x^{2}y^{2}+y^{4},  x^{2}y^{2}+y^{2}z^{2},  z^{3},  xz,  x^{3}+xy^{2})\\
    \hline B_{11},B_{12},B_{14} & (x^{2}y^{2}+y^{4}+y^{3}z,  x^{2}y^{2}+y^{2}z^{2},  z^{3},  xz,  x^{3}+xy^{2})\\
    \hline B_{13},B_{15},B_{16} & (y^{4},  x^{2}y^{2}+y^{2}z^{2},  z^{3},  xz,  x^{3}+xy^{2})\\
    \hline B_{17} & (y^{4}+y^{3}z,  x^{2}y^{2}+y^{2}z^{2},  z^{3},  xz,  x^{3}+xy^{2})\\
    \hline B_{21} & (x^{2}y^{2}+y^{2}z^{2},  xy^{3}+y^{4},  y^{2}z+z^{3},  xz,  x^{3})\\ \hline
\end{array}$$

Let $V$ be the vector space generated by $x,y,z$ over $\Z_2$.
In order to show that there is no pair $(M,N)$ such that $H^{\ast}(M;\Z_2)\cong H^{\ast}(N;\Z_2)$, we shall check that there is no linear bijective map $V$ to $V$ which sends $I_A$ to $I_B$ for any $A_i$ and $B_j$.

For $\gamma\in\{x,y,z,x+y,x+z,y+z,x+y+z\}$, we define the \emph{codimension} of $\gamma$, denoted by $\codim \gamma$, as the minimum order of $\delta$ such that $\gamma \delta \in I$ and the \emph{order} of $\gamma$, denoted by $\ord\gamma$, as the minimum of $k$ such that $\gamma^k \in I$.

Here are the lists of codimensions and orders of all linear terms corresponding to each $A_i$.
\begin{center}
\begin{tabular}{|l|c|c|c|c|c|c|c|}
  \hline
  $\codim$ & $x$ & $y$ & $z$ & $x+y$ & $x+z$ & $y+z$ & $x+y+z$ \\ \hline
  $A_1$ &1&2&1&3&3&2&2\\ \hline
  $A_2$ &1&2&1&3&3&2&3\\ \hline
  $A_4$ &1&2&1&3&3&2&2\\ \hline
  $A_8$ &1&2&1&3&3&2&2\\ \hline
  $A_9$ &1&2&1&2&3&2&2\\ \hline
  $A_{10}$ &1&2&1&3&2&2&2\\ \hline
  $A_{13}$ &1&2&1&2&3&2&2\\ \hline
  $A_{14}$ &1&2&1&3&2&2&3\\ \hline
  $A_{15}$ &1&2&1&3&2&2&3\\ \hline
  $A_{17}$ &1&2&1&3&2&2&3\\ \hline
  $A_{21}$ &1&2&1&3&2&2&3\\
  \hline
\end{tabular}
\end{center}

\begin{center}
\begin{tabular}{|l|c|c|c|c|c|c|c|}
  \hline
  $\ord$ & $x$ & $y$ & $z$ & $x+y$ & $x+z$ & $y+z$ & $x+y+z$ \\ \hline
  $A_1$ &3&5&6&5&6&3&4\\ \hline
  $A_2$ &6&6&5&6&6&3&5\\ \hline
  $A_4$ &6&5&6&6&5&3&6\\ \hline
  $A_8$ &5&5&5&4&4&3&5\\ \hline
  $A_9$ &4&6&6&5&6&6&6\\ \hline
  $A_{10}$ &6&6&6&4&6&6&5\\ \hline
  $A_{13}$ &4&5&6&6&6&5&6\\ \hline
  $A_{14}$ &6&5&3&4&5&5&4\\ \hline
  $A_{15}$ &5&5&3&5&5&5&5\\ \hline
  $A_{17}$ &6&6&3&6&6&6&6\\ \hline
  $A_{21}$ &6&6&3&5&6&6&6\\
  \hline
\end{tabular}
\end{center}

Here are the lists of codimensions and orders of all linear terms corresponding to each $B_i$.
\begin{center}
\begin{tabular}{|l|c|c|c|c|c|c|c|}
  \hline
  $\codim$ & $x$ & $y$ & $z$ & $x+y$ & $x+z$ & $y+z$ & $x+y+z$ \\ \hline
  $B_1$ &1&3&1&3&2&2&2\\ \hline
  $B_2$ &1&3&1&2&2&2&2\\ \hline
  $B_4$ &1&3&1&2&2&2&2\\ \hline
  $B_5$ &1&3&1&3&2&2&2\\ \hline
  $B_6$ &1&3&1&2&2&2&2\\ \hline
  $B_9$ &1&3&1&3&2&2&2\\ \hline
  $B_{10}$ &1&3&1&2&2&3&2\\ \hline
  $B_{11}$ &1&3&1&2&2&3&2\\ \hline
  $B_{13}$ &1&3&1&2&2&3&2\\ \hline
  $B_{17}$ &1&3&1&2&2&3&2\\ \hline
  $B_{21}$ &1&3&1&3&2&2&2\\
  \hline
\end{tabular}
\end{center}

\begin{center}
\begin{tabular}{|l|c|c|c|c|c|c|c|}
  \hline
  $\ord$ & $x$ & $y$ & $z$ & $x+y$ & $x+z$ & $y+z$ & $x+y+z$ \\ \hline
  $B_1$ &3&6&5&6&5&4&4\\ \hline
  $B_2$ &6&5&6&5&5&4&5\\ \hline
  $B_4$ &5&6&5&4&4&4&6\\ \hline
  $B_5$ &3&5&5&6&5&6&5\\ \hline
  $B_6$ &6&5&6&4&5&5&5\\ \hline
  $B_9$ &3&4&5&4&5&6&6\\ \hline
  $B_{10}$ &5&6&3&4&5&6&4\\ \hline
  $B_{11}$ &5&5&3&6&5&6&5\\ \hline
  $B_{13}$ &5&4&3&6&5&4&6\\ \hline
  $B_{17}$ &5&6&3&5&5&5&6\\ \hline
  $B_{21}$ &3&6&5&5&5&6&6\\
  \hline
\end{tabular}
\end{center}

Note that the multi-sets of codimensions and orders of all elements in $V$ must be invariant under a linear map which sends $I_A$ to $I_B$.
Let us firstly consider codimensions.
Since codimensions of only $x$ and $z$ are all 1 in $I_A$ and $I_B$, the linear map should sends $\{x, z\}$ to $\{x, z\}$.
Then, $x+z \mapsto x+z$ via the linear map.

For $A_1,A_2,A_4,A_8,A_9,A_{13}$,
since the codimension of $x+z$ in $I$ is different from that in $I_B$ for any $B_i$, there is not a such linear map between $I_A$ and $I_B$.
For $A_{10}$, while the orders of $x$, $z$, and $x+z$ are all 6 in $I_A$, there is no $B_i$ satisfying the orders of $x$, $z$, and $x+z$ are all 6.
Thus, there is not a such linear map between $I_A$ and $I_B$.
For $A_{14}$, $A_{17}$ and $A_{21}$, while the orders of $x$ and  $z$ are 6, 3, respectively, there is no $B_i$ satisfying the set of orders of $x$ and $z$ is $\{6,3\}$ in $I_B$.
Thus, there is not a such linear map between $I_A$ and $I_B$.
Finally, for $A_{15}$, the order of each element in $V$ is either $3$ or $5$ in $I_A$. However, for any $B_i$, there is at least one element in $V$ whose order is $6$ in $I_B$.
Therefore, there is no such linear map between $I_A$ and $I_B$ for any pair of $A_i$ and $B_j$, which proves the theorem.
\end{proof}

\bigskip

\bibliographystyle{amsplain}

\begin{thebibliography}{10}
\bibitem{Bosio_A-rigid}
Fr\'{e}d\'{e}ric Bosio, \emph{Combinatorially rigid simple polytopes with $d +
  3$ facets}, preprint, arXiv:1511.09039, 2015.

\bibitem{Buchstaber2008}
Victor Buchstaber, \emph{Lectures on toric topology}, In: Proceedings of Toric
  Topology Workshop KAIST 2008, vol.~10, Trends in Math.–New Series, no.~1,
  Information Center for Mathematical Sciences, KAIST, 2008, pp.~1--64.

\bibitem{BEMPP}
Victor Buchstaber, Nikolay Erokhovets, Mikiya Masuda, Taras Panov, and
  Seonjeong Park, \emph{Cohomological rigidity of manifolds defined by
  right-angled $3$-dimensional polytopes}, preprint, arXiv:1610.07575, 2016.

\bibitem{Buchstaber-Panov2015}
Victor~M. Buchstaber and Taras~E. Panov, \emph{Toric topology}, Mathematical
  Surveys and Monographs, vol. 204, American Mathematical Society, Providence,
  RI, 2015. 

\bibitem{Choi2013}
Suyoung Choi, \emph{Different moment-angle manifolds arising from two polytopes
  having the same bigraded {B}etti numbers}, Algebr. Geom. Topol. \textbf{13}
  (2013), no.~6, 3639--3649. 

\bibitem{Choi-Kim2011}
Suyoung Choi and Jang~Soo Kim, \emph{Combinatorial rigidity of 3-dimensional
  simplicial polytopes}, Int. Math. Res. Not. IMRN (2011), no.~8, 1935--1951.


\bibitem{Choi-Panov-Suh2010}
Suyoung Choi, Taras Panov, and Dong~Youp Suh, \emph{Toric cohomological
  rigidity of simple convex polytopes}, J. Lond. Math. Soc. (2) \textbf{82}
  (2010), no.~2, 343--360. 

\bibitem{Davis-Januszkiewicz1991}
Michael~W. Davis and Tadeusz Januszkiewicz, \emph{Convex polytopes, {C}oxeter
  orbifolds and torus actions}, Duke Math. J. \textbf{62} (1991), no.~2,
  417--451.

\bibitem{Erokhovets2011}
N.~Yu. Erokhovets, \emph{Moment-angle manifolds of simple {$n$}-dimensional
  polytopes with {$n+3$} facets}, Uspekhi Mat. Nauk \textbf{66} (2011),
  no.~5(401), 187--188. 

\bibitem{Garrison-Scott2003}
Anne Garrison and Richard Scott, \emph{Small covers of the dodecahedron and the
  120-cell}, Proc. Amer. Math. Soc. \textbf{131} (2003), no.~3, 963--971.


\bibitem{Gretenkort-Kleinschmidt-Sturmfels1990}
J\"org Gretenkort, Peter Kleinschmidt, and Bernd Sturmfels, \emph{On the
  existence of certain smooth toric varieties}, Discrete Comput. Geom.
  \textbf{5} (1990), no.~3, 255--262. 

\bibitem{Grunbaum2003}
Branko Gr\"unbaum, \emph{Convex polytopes}, second ed., Graduate Texts in
  Mathematics, vol. 221, Springer-Verlag, New York, 2003, Prepared and with a
  preface by Volker Kaibel, Victor Klee and G\"unter M. Ziegler. 

\bibitem{Lopez1989}
Santiago L\'opez~de Medrano, \emph{Topology of the intersection of quadrics in
  {${\bf R}^n$}}, Algebraic topology ({A}rcata, {CA}, 1986), Lecture Notes in
  Math., vol. 1370, Springer, Berlin, 1989, pp.~280--292. 

\bibitem{Masuda-Suh2008}
Mikiya Masuda and Dong~Youp Suh, \emph{Classification problems of toric
  manifolds via topology}, Toric topology, Contemp. Math., vol. 460, Amer.
  Math. Soc., Providence, RI, 2008, pp.~273--286.

\bibitem{Park2015_thesis}
Kyoungsuk Park, \emph{Combinatorics of coxeter groups with permutation tableaux
  and cohomological rigidity of simple polytopes}, Ph.D. Thesis, Ajou
  university, 2015, pp.~1--110.
\end{thebibliography}

\end{document}